\documentclass[12pt]{amsart}
\usepackage{amsmath, amssymb, amsthm}
\usepackage{paralist, xcolor, tikz, hyperref}

\usepackage[margin=1in,letterpaper,portrait]{geometry}

\usepackage{caption}
\usepackage{enumitem}
\usepackage{comment}
\usepackage{tabularx}
\usepackage{ytableau}

\newcommand{\nn}{n^2}

\newcommand{\DT}{\mathrm{DT}}
\newcommand{\ms}{S}

\newcommand{\qbin}[2]{\left[ \begin{array}{c} #1 \\ #2 \end{array}\right]}
\newcommand{\maj}{\mathrm{maj}}

\newcommand{\Des}{\mathrm{Des}}
\renewcommand{\mod}{\hspace{1 mm}\mathrm{mod} \hspace{1 mm}}
\newtheorem{conjecture}{Conjecture}[section]
\newtheorem{theorem}[conjecture]{Theorem}
\newtheorem{corollary}[conjecture]{Corollary}
\newtheorem{lemma}[conjecture]{Lemma}

\theoremstyle{definition}

\newtheorem{example}[conjecture]{Example}

\newtheorem{definition}[conjecture]{Definition}
\newtheorem*{question*}{Question}

\usepackage{tkz-graph}
\usetikzlibrary[positioning, patterns] 
\usetikzlibrary{arrows}
\usetikzlibrary{shapes.geometric, calc, math}
\usepackage{tikz, tikz-cd} 

\tikzstyle{bsq}=[rectangle, draw, thick, minimum width=1cm, minimum height=1cm] 
\tikzstyle{bver}=[rectangle, draw, thick, minimum width=1cm, minimum height=2cm]
\tikzstyle{bhor}=[rectangle, draw, thick, minimum width=2cm, minimum height=1cm]

\tikzstyle{bverg}=[rectangle, draw, thick, minimum width=1.15cm, minimum height=2.3cm]
\tikzstyle{bhorg}=[rectangle, draw, thick, minimum width=2.3cm, minimum height=1.15cm]

\newdimen\squareheight \squareheight=15pt
\newdimen\squarewidth  \squarewidth=30pt 
\newdimen\thickness    \thickness=0.4pt

\def\square#1{
  \hbox{
    \vrule width \thickness
    \vbox to \squareheight{
      \hrule height \thickness
      \vss
      \hbox to \squarewidth{\hss #1\hss}
      \vss
      \hrule height \thickness
    }
    \unskip\vrule width \thickness
  }
  \kern-\thickness
}

\def\vsquare#1{\vbox{\square{$#1$}}\kern-\thickness}

\title{Cyclic sieving for a class of rectangular domino tableaux}

\author[Laura Colmenarejo]{Laura Colmenarejo$^\dagger$}
\address{Department of Mathematics, North Carolina State University, Raleigh, NC, USA}
\thanks{$^\dagger$Research partially supported by the Simons Foundation.}
\email{lcolmen@ncsu.edu}

\author[Bridget Eileen Tenner]{Bridget Eileen Tenner$^\ddagger$}
\address{Department of Mathematical Sciences, DePaul University, Chicago, IL, USA}
\email{bridget@math.depaul.edu}
\thanks{$^\ddagger$Research partially supported by NSF Grant DMS-2054436 and by a University Research Council Competitive Research Leave from DePaul University.}

\author[Camryn E. Thompson]{Camryn E. Thompson}
\address{Department of Mathematics, North Carolina State University, Raleigh, NC, USA}
\email{cgrey@ncsu.edu}

\begin{document}

\begin{abstract}
The cyclic sieving phenomenon (CSP) provides valuable data about symmetry classes of cyclic actions, and has applications to representation theory. 
In this paper, we enumerate domino tableaux of shape $2\times n$, and use this result to prove a new CSP on these objects. We then enumerate the rectangular domino tableaux of any dimensions, 
and conjecture a more general CSP on rectangular domino tableaux. As a consequence of the enumerative results, we obtain several identities involving Fibonacci and Catalan numbers.
\end{abstract}

\maketitle

\section{Introduction}  
In a seminal 2004 paper, Reiner, Stanton, and White~\cite{OGCSP} introduced what they called the \emph{cyclic sieving phenomenon (CSP)}. This was a recognition that there are some sets for which there is a substantial relationship between a natural collection of symmetries of the set and a natural generating function for the set. In particular, given a cyclic action on a finite set $X$, one might seek to determine the sizes of the fixed point sets under this action, which gives data that fully classifies the cyclic action. The CSP identifies a polynomial that produces the same data when evaluated at appropriate roots of unity. Although this scenario might initially seem far-fetched, it turns out to underlie 
a substantial number of combinatorial objects, and it has been studied widely in the past two decades.

The cyclic sieving phenomenon is a generalization of Stembridge's $q=-1$ phenomenon, which Stembridge introduced in his 1994 paper when he unified the proofs for enumerating the symmetry classes of plane partitions by evaluating a $q$-analogue at $q=-1$~\cite{q-1}. The cyclic sieving phenomenon extends this idea by evaluating a polynomial at all appropriate roots of unity, and, in turn, fully classifying the symmetry classes of a cyclic action on a finite set. The CSP requires three components: a finite set $X$, a cyclic group $C=\langle g \rangle$ of order $n$ that acts on $X$, and a polynomial $f(q)$ with coefficients in $\mathbb{N}$. 
We say that the triple $(X, C, f(q))$ exhibits the \emph{cyclic sieving phenomenon} if, for all $k$, 
\begin{equation}\label{CSP def}
   \#\{x\in X\mid g^k(x)=x\}= f(e^{2\pi ik/n}). 
\end{equation} 
In other words, the triple exhibits the CSP if, for any $g^k\in C$, the number of elements in $X$ fixed by $g^k$ under the cyclic action is equal to the evaluation of $f(q)$ at the $k$th $n$th-root of unity. 

To gain more insight into this equality, notice that evaluating the polynomial $f(q)$ at $q=1$ recovers $\#X$, because every element in $X$ is fixed by $g^n$. Furthermore, when $\#C=2$, the CSP specializes to the $q=-1$ phenomenon~\cite{OGCSP}. It is common to consider the polynomial $f(q)$ to be a generating function counting the elements of $X$ according to some statistic, and to consider an action of the cyclic group on $X$ described in terms of that same statistic.

The CSP has been studied with great interest since it was defined, being exhibited in various settings, each with their own applications. For instance, a CSP has been shown for important classes of tableaux~\cite{RTCSP}, dissections of a polygon~\cite{dissect}, Kac--Moody algebras~\cite{KMA}, finite Grassmannians and flag varieties~\cite{GFV}, parking functions~\cite{parking}, and rooted plane trees~\cite{rpt}, to name a few.

In this paper, we establish a new instance of the cyclic sieving phenomenon on the set of domino tableaux of shape $\lambda=(n,n)$, denoted $\DT(n^2)$. In Section~\ref{sec:Background}, we outline the relevant background of domino tableaux and state some examples of the CSP. In Section~\ref{count dtn2}, we consider $\DT(\nn)$ and prove the enumeration for this set. With an initial formula for $\#\DT(\nn)$, we describe an alternative counting method in Section~\ref{cat identity} that will yield new identities involving Catalan and Fibonacci numbers. In Section~\ref{all CSP}, we determine the cyclic action and polynomial needed to compose a CSP triple, and subsequently state and prove our new instance of the CSP. We broaden our scope to the set of all rectangular domino tableaux, $\DT(n^m)$, in Section~\ref{count rectangles}, where we give additional enumeration results. Finally, we conclude in Section~\ref{conclusion} by conjecturing a new CSP for the set $\DT(n^m)$, and pose other relevant open problems.

\section{Background}\label{sec:Background}

\subsection{Domino tableaux}\label{subsec: domino}

A \emph{partition} $\lambda=(\lambda_1,\ldots, \lambda_\ell)$ is a weakly decreasing sequence of positive integers. We refer to $\lambda_i$ as a \emph{part} of the partition and write $\lambda \vdash n$ or $|\lambda|=n$ if $\sum_i \lambda_i=n$. The \emph{Ferrers diagram} of $\lambda$ is the arrangement of $\lambda_i$ boxes in row $i$. We 
use English notation, which means that the rows are left justified, with row 1 at the top. 

\begin{definition}
Given a partition $\lambda$, a \emph{domino tiling of shape $\lambda$} is a tiling of the Young diagram of $\lambda$ by $2 \times 1$ or $1 \times 2$ ``dominoes.''  
\end{definition}

Note that if $\lambda$ is a partition of an odd integer, then there are no domino tilings of shape $\lambda$. Thus, we only define domino tableaux for shapes that are partitions of even integers.

\begin{definition}
    A \emph{domino tableau of shape $\lambda$} is a domino tiling of the Young diagram of $\lambda \vdash 2n$ in which the $n$ dominoes are labeled with the distinct integers $[n]:=\{1,2,\ldots,n\}$ so that the domino labels across each row and down each column are strictly increasing. We denote the set of domino tableaux of shape $\lambda$ by $\DT(\lambda)$.  
\end{definition} 

\begin{example} \label{domtab}
Below are domino tableaux of shapes $(5,5,3,3,2)$ and $(9,9)$, respectively, each having 9 dominoes.
    $$
    \begin{matrix} \scalebox{0.5}{
\begin{tikzpicture}[node distance=0 cm,outer sep = 0pt]
        \node[bver] (1) at ( 0,   0) {\bf \Large 1};
        \node[bver] (2) [right = of 1] {\bf \Large 2};
        \node[bhor] (3) at (0.5, -1.5) {\bf \Large 3};         
        \node[bhor] (4) at (2.5, 0.5) {\bf \Large 4};        
        \node[bver] (5) at (0, -3) {\bf \Large 5}; 
        \node[bhor] (6) [below = of 4] {\bf \Large 6};
        \node[bver] (7) [right = of 5] {\bf \Large 7};
        \node[bver] (8) at (2, -2) {\bf \Large 8};
        \node[bver] (9) at (4, 0) {\bf \Large 9};
\end{tikzpicture}  } \end{matrix} \hspace{25 mm} 
\begin{matrix} \scalebox{0.5}{
\begin{tikzpicture}[node distance=0 cm,outer sep = 0pt]
        \node[bver] (1) at ( 0,   0) {\bf \Large 1};
        \node[bver] (2) [right = of 1] {\bf \Large 2};
        \node[bhor] (3) at (2.5, 0.5) {\bf \Large 3};         
        \node[bhor] (4) [below = of 3] {\bf \Large 4};        
        \node[bhor] (5) [right = of 3] {\bf \Large 5}; 
        \node[bhor] (6) [right = of 5] {\bf \Large 6};
        \node[bhor] (7) [below = of 5] {\bf \Large 7};
        \node[bhor] (8) [below = of 6] {\bf \Large 8};
        \node[bver] (9) at (8,0) {\bf \Large 9};
\end{tikzpicture}  } \end{matrix}$$
\end{example}

In this paper, we focus our study on the set of domino tableaux of rectangular shape $m \times n$; that is, of partition shape $(n^m)$. In particular, for domino tableaux of shape $2\times n$, like the second shape of Example~\ref{domtab}, we say that a domino is a \emph{top horizontal domino} if it is horizontal and in row 1, and a \emph{bottom horizontal domino} if it is horizontal and in row 2. Each top horizontal domino has a bottom horizontal domino in alignment below it, as forced by the overall shape $2 \times n$. We call such a pair of aligned horizontal dominoes a \emph{stack}. A domino tableau $D \in \DT(\nn)$ has a \emph{$k$-stack} if there are $k$ consecutive horizontal stacks in $D$ and $k$ is maximal, and a \emph{stacking} is a $k$-stack for some unspecified value of $k$.
For example, the second tableau in Example~\ref{domtab} has one stacking, and it is a 3-stack. 

Given a partition $\lambda$, we define an \emph{increasing Young tableau (increasing YT) of shape $\lambda$} as a filling of the Ferrers diagram of $\lambda$ in English notation, with increasing numbers in the rows (from left to right) and in the columns (from top to bottom). A \emph{standard YT of shape $\lambda$} is an increasing YT where the labels in the filling are exactly $[n]$ where $n = |\lambda|$.

There is a bijection between domino tableaux and pairs of increasing YT. 
Stanton and White originally stated this bijection as a bijection between rim hook tableaux and pairs of increasing Young tableaux~\cite{rimhook}. Carr\'e and Leclerc later simplified the bijection by looking at the set of domino tableaux, a subset of rim hook tableaux, and using the diagonals of these tableaux to describe the algorithm~\cite{DTSYT}. 

We define $\Gamma$ to be the map that sends a domino tableau $D$ of shape $\lambda$ to a pair of increasing YT $(\mu, \nu)$ so that their labeling sets are disjoint and their union is exactly the labeling set of $D$. The map $\Gamma$ is defined using the \emph{2-quotient} of $\lambda$, which we define as follows.

\begin{definition}[{\cite{DTSYT}}]
Let $\ell$ be the number of parts in $\lambda$. Define $L$ to be the $\ell$-tuple whose $i$th element is $\lambda_i+\ell-i$. Next, define $M$ to be the $\ell$-tuple obtained from $L$ by replacing the even numbers from right to left with $0,2,4,\ldots$ and the odd numbers with $1, 3,5,\ldots$ successively. Then $\mu$ is the partition obtained by subtracting from the odd parts of $L$ the corresponding parts of $M$ and dividing by two. The partition $\nu$ is found using the same process on even parts. The ~\emph{2-quotient} of $\lambda$ is the pair $(\mu, \nu)$. 
\end{definition}

\begin{example} Consider $\lambda= (5, 5, 3, 3, 2)$. Then $L= (9, 8, 5, 4, 2)$ and $M=(3, 4, 1, 2, 0)$. Subtracting the odd (even) parts of $L$ by the corresponding parts of $M$ yields 
$$\mu=\frac{1}{2}[(9, 5)-(3, 1)]=(3, 2) \hspace{5 mm} \text{and} \hspace{5 mm} \nu=\frac{1}{2}[(8,4,2)-(4,2,0)]=(2,1,1).$$
\end{example}

To obtain a pair of increasing tableaux from a domino tableau under $\Gamma$, one starts by considering the even-positioned diagonals of the domino tableau. In particular, we track how each domino is intersected by a diagonal and classify them according to the following two types. 
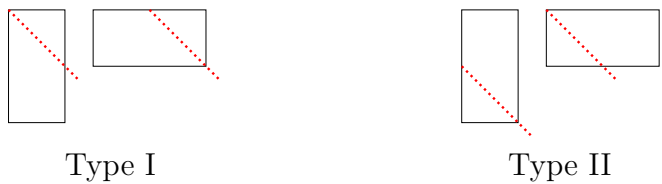
\begin{eqnarray*}
\begin{array}{ccc}
\scalebox{0.75}{
\begin{tikzpicture}
    \draw (0,0) rectangle (1,-2);  
    \draw (1.5,-1) rectangle (3.5,0);
    \draw[dotted, red, very thick] (0,0) -- (1.25,-1.25);
    \draw[dotted, red, very thick] (2.5,0) -- (3.75,-1.25);
    \draw[dotted, white, very thick] (1.15,-2.15) -- (1.25,-2.25);
\end{tikzpicture}} & \qquad\qquad\qquad & 
\scalebox{0.75}{\begin{tikzpicture}
    \draw (0,0) rectangle (1,-2);  
    \draw (1.5,-1) rectangle (3.5,0);
    \draw[dotted, red, very thick] (0,-1) -- (1.25,-2.25);
    \draw[dotted, red, very thick] (1.5,0) -- (2.75,-1.25);
\end{tikzpicture}} \\
\text{Type I} & & \text{Type II}
\end{array}
\end{eqnarray*}

\begin{definition}\label{defn:Gamma}
    Let $\Gamma$ be the map that sends a domino tableau $D$ of shape $\lambda$ to a pair of increasing YT whose shapes are the 2-quotient $(\mu,\nu)$ of $\lambda$ labeled as follows: 
    the increasing YT of shape $\mu$ is obtained by deleting all dominoes of Type II on each diagonal and sliding the remaining dominoes up their diagonals to create an increasing YT, and the increasing YT of shape $\nu$ is obtained by deleting all dominoes of Type I and sliding the remaining dominoes up their diagonals to create an increasing YT.
\end{definition}

\begin{example} Applying $\Gamma$ to the first domino tableau from Example~\ref{domtab} yields the following pair of increasing tableaux.  
\begin{eqnarray*}
\newdimen\squarewidth  \squarewidth=15pt
\begin{tikzpicture}
\node (dt) at (0,0) {
\begin{tikzpicture}
    \draw (0,0) rectangle (0.5,-1);
    \node at (0.25, -0.5) {$\mathbf{1}$};
    \draw (0.5,0) rectangle (1,-1);
    \node at (0.75, -0.5) {$\mathbf{2}$};
    \draw (1,0) rectangle (2,-0.5);
    \node at (1.5, -0.25) {$\mathbf{4}$};
    \draw (1,-0.5) rectangle (2,-1);
    \node at (1.5, -0.75) {$\mathbf{6}$};
    \draw (2,0) rectangle (2.5,-1);
    \node at (2.25, -0.5) {$\mathbf{9}$};
    \draw (0,-1) rectangle (1,-1.5);
    \node at (0.5, -1.25) {$\mathbf{3}$};
    \draw (1,-1) rectangle (1.5,-2);
    \node at (1.25, -1.5) {$\mathbf{8}$};
    \draw (0,-1.5) rectangle (0.5,-2.5);
    \node at (0.25, -2) {$\mathbf{5}$};
    \draw (0.5,-1.5) rectangle (1,-2.5);
    \node at (0.75, -2) {$\mathbf{7}$};
    \draw[dotted, red, very thick] (0,0) -- (2.25,-2.25);
    \draw[dotted, red, very thick] (1,0) -- (2.75,-1.75);
    \draw[dotted, red, very thick] (2,0) -- (2.75,-0.75);
    \draw[dotted, red, very thick] (0,-1) -- (1.75,-2.75);
     \draw[dotted, red, very thick] (0,-2) -- (0.75,-2.75);
\end{tikzpicture}}; 
\node (diag-mu) at (5,2) {
\begin{tikzpicture}
    \draw[dotted, red,  thick] (0,0) -- (1.75,-1.75);
    \draw[dotted, red,  thick] (1,0) -- (2.25,-1.25);
    \draw[dotted, red,  thick] (2,0) -- (2.75,-0.75);
    \draw[dotted, red,  thick] (0,-1) -- (0.75,-1.75);
    \node at (0.25, -0.25) {$\mathbf{1}$};
    \node at (1.25, -0.25) {$\mathbf{6}$};
    \node at (2.25, -0.25) {$\mathbf{9}$};
    \node at (1.25, -1.25) {$\mathbf{8}$};
    \node at (0.25, -1.25) {$\mathbf{7}$};
\end{tikzpicture}};
\node (diag-nu) at (5,-2) {
\begin{tikzpicture}
    \draw[dotted, red, thick] (0,0) -- (1.25,-1.25);
    \draw[dotted, red, thick] (1,0) -- (2.25,-1.25);
    \draw[dotted, red, thick] (0,-1) -- (0.75,-1.75);
    \draw[dotted, red,  thick] (0,-2) -- (0.75,-2.75);
   \node at (0.25, -0.25) {$\mathbf{2}$};
    \node at (1.25, -0.25) {$\mathbf{4}$};
    \node at (0.25, -1.25) {$\mathbf{3}$};
    \node at (0.25, -2.25) {$\mathbf{5}$};
\end{tikzpicture}};
\node (mu) at (8.5,2) {
\begin{tikzpicture}
    \draw (0,0) rectangle (0.5,-0.5);
    \draw (0.5,0) rectangle (1,-0.5);
    \draw (1,0) rectangle (1.5,-0.5);
    \draw (0,-0.5) rectangle (0.5,-1);
    \draw (0.5,-0.5) rectangle (1,-1);
    \node at (0.25, -0.25) {$\mathbf{1}$};
    \node at (0.75, -0.25) {$\mathbf{6}$};
    \node at (1.25, -0.25) {$\mathbf{9}$};
    \node at (0.25, -0.75) {$\mathbf{7}$};
    \node at (0.75, -0.75) {$\mathbf{8}$};
    \node at (2,-0.5) {$=\mu$};
\end{tikzpicture}
};  
\node (nu) at (8.5,-2) {\begin{tikzpicture}
    \draw (0,0) rectangle (0.5,-0.5);
    \draw (0.5,0) rectangle (1,-0.5);
    \draw (0,-0.5) rectangle (0.5,-1);
    \draw (0,-1) rectangle (0.5,-1.5);
    \node at (0.25, -0.25) {$\mathbf{2}$};
    \node at (0.75, -0.25) {$\mathbf{4}$};
    \node at (0.25, -0.75) {$\mathbf{3}$};
    \node at (0.25, -1.25) {$\mathbf{5}$};
    \node at (1.5,-0.75) {$=\nu$};
\end{tikzpicture}};
\draw (dt) edge[->] node[sloped,above] {Type I} (diag-mu);
\draw (dt) edge[->] node[sloped,below] {Type II} (diag-nu);
\draw (diag-mu) edge[->]  (mu);
\draw (diag-nu) edge[->] (nu);
\end{tikzpicture}
\end{eqnarray*}

\end{example}

\begin{theorem}[\cite{DTSYT}]\label{DTSYT} 
The map $\Gamma$ is a bijection.
\end{theorem}

We define two statistics on domino tableaux by extending the notion of descent and major index from standard Young tableaux. 

\begin{definition}\label{defn:descent}
A domino tableau $D$ has a \emph{descent} at position $i$ if the northeast-most cell of the domino labeled $i+1$ is in a row that is below the northeast-most cell of the domino labeled $i$. The \emph{descent set} of $D$, denoted $\Des(D)$, 
is the set of all descents of $D$.  
The \emph{major index} of a domino tableau is defined in the usual way: 
\begin{equation*}
        \maj(D)= \displaystyle \sum_{i \in \Des(D)} i .
\end{equation*}
\end{definition}

\begin{example}
For $D= \begin{matrix} \scalebox{0.5}{
\begin{tikzpicture}[node distance=0 cm,outer sep = 0pt]
        \node[bhor] (1) at ( 0,   0) {\bf \Large 1};
        \node[bhor] (2) [below = of 1] {\bf \Large 2};
        \node[bver] (3) at (1.5, -.5) {\bf \Large 3};         
        \node[bhor] (4) at (3, 0) {\bf \Large 4};        
        \node[bhor] (5) [below = of 4] {\bf \Large 5}; 
\end{tikzpicture}  } \end{matrix}$, we have $\Des(D) =\{ 1, 4\}$ and $\maj(D) = 1+4=5$.
\end{example}

\subsection{Examples of cyclic sieving}\label{subsec:CSP}
We present here some known instances of the CSP that are relevant to our study, and we state them in the most convenient form for this paper. 
The first example is a CSP on the set of $k$-element subsets of $[n]$.

\begin{theorem}[{\cite[Theorem 1.1(b)]{OGCSP}}]\label{subset CSP} 
Let $X_n(k)$ be the set of $k$-element subsets of $[n]$ and let $C=\langle g\rangle$ be the cyclic group of order $n$. Consider the action of $C$ on $X_n(k)$ given by 
$g \hspace{0.05cm} \raisebox{0.2ex}{$\scriptscriptstyle\bullet$}\hspace{0.05cm}  x = (x+1)\mod n$.  
Then the triple $(X_n(k), C, f(q))$ exhibits the cyclic sieving phenomenon, where $f(q)=\qbin{n}{k}_q$ is the $q$-analogue of the binomial coefficient. 
\end{theorem}

Our second example comes from a CSP on finite Coxeter groups. We present here the version for the set $W_{n,k}$ of binary words with length $n$ and $k$ zeros, for which the cyclic shift action is defined by taking the first letter in the word and moving it to the end of the word. 

\begin{theorem}[{\cite[Theorem 1.6]{OGCSP}}]
    Let $C$ be the cyclic group of order $n$ acting on $W_{n,k}$ by the cyclic shift. Then the triple $(W_{n,k}, C, f(q))$ exhibits the cyclic sieving phenomenon, where $f(q)= \displaystyle \sum_{w\in W_{n,k}} q^{\maj(w)}$.
\end{theorem} 

The representatives of the orbits in $W_{n,k}$ under this cyclic action are the Lyndon words of length $n$ with $k$ zeros (see Section~\ref{CSP Action} for more details). 

Finally, we present an instance of the CSP on standard Young tableaux of rectangular shape, studied by Rhoades using promotion as the cyclic action. We refer the reader to~\cite{RTCSP} for more details on the promotion action.
\begin{theorem}[{\cite[Theorem 1.3]{RTCSP}}]\label{RTCSP}
Let $\lambda \vdash n$ be a rectangular partition, let $X_n(\lambda)$ be the set of standard Young tableaux of shape $\lambda$, and let $C=\mathbb{Z}_n$ act on $X_n(\lambda)$ by promotion. Then the triple $(X_n(\lambda), C, f(q))$ exhibits the cyclic sieving phenomenon, where $f(q)= \dfrac{[n]!_q}{\prod_{(i,j)\in \lambda} [h_{ij}]_q}$ is the $q$-analogue of the hook length formula.
\end{theorem}
Rhoades proved this and similar theorems using Kazhdan-Lusztig theory, for which the tableaux must have a rectangular shape. 

\section{Enumerating $\DT(\nn)$}\label{count dtn2}
Our initial goal is to determine $\#\DT(\nn)$, the number of domino tableaux of shape $2 \times n$. It is known that the number of domino tilings of $\lambda=(\nn)$ is the $n$th Fibonacci number; however, a given tiling might admit more than one labeling.  

\begin{example} \label{F}
    Below are two domino tableaux with the same underlying tiling of $\lambda=(9^2)$.
    $$\scalebox{0.5}{
\begin{tikzpicture}[node distance=0 cm,outer sep = 0pt]
        \node[bver] (1) at ( 0,   0) {\bf \Large 1};
        \node[bver] (2) [right = of 1] {\bf \Large 2};
        \node[bhor] (3) at (2.5, 0.5) {\bf \Large 3};         
        \node[bhor] (4) [below = of 3] {\bf \Large 4};        
        \node[bhor] (5) [right = of 3] {\bf \Large 5}; 
        \node[bhor] (6) [below = of 5] {\bf \Large 6};
        \node[bver] (7) at (6, 0) {\bf \Large 7};
        \node[bhor] (8) at (7.5, 0.5) {\bf \Large 8};
        \node[bhor] (9) [below = of 8] {\bf \Large 9};
\end{tikzpicture}  } \qquad
\scalebox{0.5}{
\begin{tikzpicture}[node distance=0 cm,outer sep = 0pt]
        \node[bver] (1) at ( 0,   0) {\bf \Large 1};
        \node[bver] (2) [right = of 1] {\bf \Large 2};
        \node[bhor] (3) at (2.5, 0.5) {\bf \Large 3};         
        \node[bhor] (4) [right = of 3] {\bf \Large 4};        
        \node[bhor] (5) [below = of 3] {\bf \Large 5}; 
        \node[bhor] (6) [below = of 4] {\bf \Large 6};
        \node[bver] (7) at (6, 0) {\bf \Large 7};
        \node[bhor] (8) at (7.5, 0.5) {\bf \Large 8};
        \node[bhor] (9) [below = of 8] {\bf \Large 9};
\end{tikzpicture}  }
$$
\end{example}

Such variation arises from a single kind of configuration in the tiling.

\begin{lemma}
    A given domino tiling of shape $(n^2)$ yields more than one domino tableaux if and only if the tiling has a $k$-stack with $k\geq 2$.
\end{lemma}

\begin{proof}
    The label of any vertical tile in a tiling of $(n^2)$ is determined by how many tiles, of either orientation, appear to its left. The labels of any $1$-stack tile are similarly forced. However, in a $k$-stack for $k \ge 2$, there is more than one increasing labeling, as demonstrated in Example~\ref{F}.
\end{proof}

Furthermore, we know exactly how many labelings there are for a stacking.

\begin{lemma} \label{Cat}
The number of labelings of a $k$-stack is $C_{k}$, the $k$th Catalan number.  
\end{lemma}

\begin{proof}
The set of domino tableaux of shape $((2k)^2)$ composed of only horizontal dominoes (i.e., the tiling is a single $k$-stack) is in bijection with the set of standard Young tableaux of shape $(k^2)$. To see that, one identifies the $1\times 2$ horizontal dominoes with the $1\times 1$ blocks of a standard Young tableau. 
Finally, the number of standard Young tableaux of shape $(k^2)$ is the $k$th Catalan number. 
\end{proof}

Perhaps one of the more interesting properties of $\DT(n^2)$ is its determination by bottom horizontal domino labels. That is, given the labels of the bottom horizontal dominoes for some domino tableau of shape $(\nn)$, the remaining dominoes must all intersect the first row. Therefore, their labels are exactly the remaining numbers and must be placed in increasing order from left to right. Thus, we have the following key result.

\begin{lemma}\label{U}
    A domino tableau of shape $\lambda=(\nn)$ is uniquely determined by the labeling of its bottom horizontal dominoes.
\end{lemma}

Note that Lemma~\ref{U} says that the underlying domino tiling is determined by this data, not just the labels of the dominoes in that tiling.

\begin{example}
    Suppose $n=5$ and the bottom horizontal labels of the domino tableau $D$ are 2 and 5. Then we must have 
$$D= \begin{matrix} \scalebox{0.5}{
\begin{tikzpicture}[node distance=0 cm,outer sep = 0pt]
        \node[bhor] (1) at ( 0,   0) {\bf \Large 1};
        \node[bhor] (2) [below = of 1] {\bf \Large 2};
        \node[bver] (3) at (1.5, -.5) {\bf \Large 3};         
        \node[bhor] (4) at (3, 0) {\bf \Large 4};        
        \node[bhor] (5) [below = of 4] {\bf \Large 5}; 
\end{tikzpicture}  } \end{matrix}.$$
\end{example}

With these observations, and the fact that a domino tableau of shape $(\nn)$ has at most $\lfloor n/2 \rfloor$ stacks, we have the tools needed to enumerate $\DT(\nn)$.

\begin{theorem} \label{BIJ}
For $n\geq 1$, $\# \DT(\nn)= \displaystyle \binom{n}{\lfloor n/2 \rfloor}$.
\end{theorem}

\begin{proof}
We will construct a bijection 
$$\Phi : \DT(\nn)\rightarrow \{ \lfloor n/2 \rfloor\text{-element subsets of } [n]\}.$$
This will prove the result because the latter set has cardinality $\binom{n}{\lfloor n/2 \rfloor}$. 

Starting with an element in $\DT(\nn)$, record the labels of the bottom horizontal dominoes. To this list, add the numbers corresponding to the smallest labels of vertical dominoes, until we have constructed an $\lfloor n/2 \rfloor$-element subset. By Lemma~\ref{U}, the information regarding the bottom horizontal dominoes is enough to uniquely determine the domino tableau, and adding vertical domino labels introduces no new information about the tableau. 

Now consider an $\lfloor n/2 \rfloor$-element subset $S$ of $[n]$. We outline an algorithm that identifies an element of $S$ as the label of a vertical domino or a bottom horizontal domino. Note that the label of a vertical domino accounts for one domino in the domino tableau, and the label of a bottom horizontal domino accounts for two dominoes, since we know it is part of a stack. Start by looking at the smallest value $m \in S$. If $m=1$, then the first domino in the corresponding domino tableau is vertical and labeled $1$. If $m \neq 1$, then $m$ is the label of a bottom horizontal domino somewhere in the domino tableau. Work through the remaining elements of $S$ in increasing order to determine whether they are the label of a vertical or bottom horizontal domino, as follows. 
Given $s \in S$, suppose we have already partitioned $\{j \in S \mid j < s\}$ into a set $H_s$ of bottom horizontal domino labels and $V_s$ of vertical labels. Appropriately positioning the dominoes with those labels requires $2|H_s| + |V_s|$ dominoes (an extra to sit atop each bottom horizontal domino). Set $x_s := s - (2|H_s| + |V_s|)$. If $x_s = 1$, then $s$ is the label of a vertical domino in the tableau, and if $x_s > 1$ then $s$ is the label of a bottom horizontal domino. This procedure determines the labels of all bottom horizontal dominoes, which uniquely determines the domino tableau by Lemma~\ref{U}.
\end{proof}

We illustrate the bijection of Theorem~\ref{BIJ} with the following example.

\begin{example} 
For $n = 8$, we have 
        $$
        \begin{matrix}  \scalebox{0.5}{
\begin{tikzpicture}[node distance=0 cm,outer sep = 0pt]
        \node[bhor] (1) at ( 0,   0) {\bf \Large 1};
        \node[bhor] (2) [below = of 1] {\bf \Large 2};
        \node[bver] (3) at (1.5, -0.5) {\bf \Large 3};         
        \node[bhor] (4) at (3, 0) {\bf \Large 4};        
        \node[bhor] (5) [below = of 4] {\bf \Large 5}; 
        \node[bver] (6) at (4.5, -0.5) {\bf \Large 6};
        \node[bver] (7) [right = of 6] {\bf \Large 7};
        \node[bver] (8) [right = of 7] {\bf \Large 8};
\end{tikzpicture}  }\end{matrix} \longleftrightarrow \{ 2, 3, 5, 6\}.$$
The subset $\{2,5\}$ uniquely determines the domino tableau, but it has fewer than $\lfloor 8/2 \rfloor = 4$ elements. Therefore, we need to add two more elements to the set, from the labels of vertical dominoes. The two smallest labels of vertical dominoes are $3$ and $6$. Hence, the full set corresponding to the domino tableau is $\{ 2, 3, 5, 6\}$. 

In the other direction, start by considering $2 \in \{2,3,5,6\} =:S$. Because $2>1$, we have that $2$ is the label of a bottom horizontal domino in the tableau. For $3 \in S$, we compute $3-(2\cdot 1+0)=1$ and conclude that $3$ is the label of a vertical domino. We have now accounted for one vertical domino and one bottom horizontal domino, so we interpret $5 \in S$ by computing $5-(2\cdot 1+1)=2$, and hence $5$ is the label of a bottom horizontal domino. Finally, we compute $6-(2\cdot 2+1)=1$ to determine that $6$ is the label of a vertical domino. Hence, this subset corresponds to a domino tableau with bottom horizontal domino labels $2$ and $5$, which uniquely yields the domino tableau shown above.
\end{example}

\section{New identities involving Fibonacci and Catalan numbers}\label{cat identity} 

In this section, we use previous results to give new identities involving Fibonacci and Catalan numbers. This is possible because domino tilings of a $2\times n$ rectangle are counted by Fibonacci numbers, and Lemma~\ref{Cat} suggests that Catalan numbers are related to the enumeration of $\DT(\nn)$, which was computed explicitly in Theorem~\ref{BIJ}.
 
When considering the recursion for the Fibonacci numbers in terms of domino tilings, $F_n$ is obtained from all tilings associated with $F_{n-1}$ by adding a vertical domino on the end, and all tilings associated with $F_{n-2}$ by adding a horizontal stack of dominoes to the end. When considering domino tableaux, we get the ones for $(\nn)$ from domino tableaux of shape $((n-1)^2)$ by adding a vertical domino labeled $n$ to the end, from domino tableaux of shape $((n-2)^2)$ by adding a horizontal stack of dominoes labeled $n-1$ and $n$, as well as additional tableaux that arise when adding a horizontal stack to $((n-2)^2)$ creates a $k$-stack with $k\geq 2$. Lemma~\ref{Cat} dictates how many labelings we obtain from this new $k$-stack.

This reasoning yields a new way of counting domino tableaux of shape $(\nn)$, which we introduce through a small example.

\begin{example}
    The shape $(5^2)$ has $F_5 = 8$ domino tilings and $\binom{5}{\lfloor 5/2 \rfloor} = 10$ domino tableaux.
    Grouping the tilings by the multisets of their stackings' widths yields four groups.

\smallskip   
\noindent \begin{tabularx}{\textwidth}{c X}
\text{Group 1: } & \scalebox{0.5}{
\parbox[c]{\hsize}
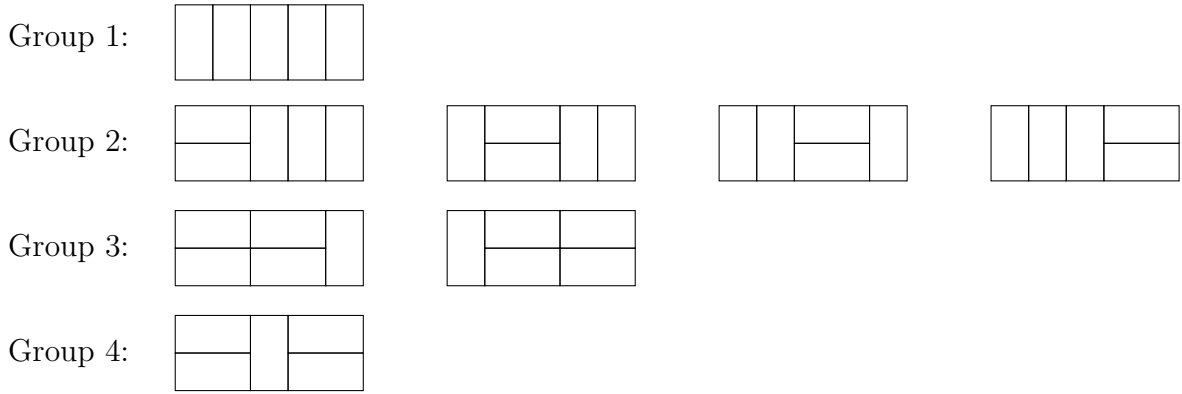
{\begin{tikzpicture}[node distance=0 cm,outer sep = 0pt]
        \node[bver] (1) at ( 0,   0) {};
        \node[bver] (2) [right = of 1]{};
        \node[bver] (3) [right = of 2]{};       
        \node[bver] (4) [right = of 3]{};       
        \node[bver] (5) [right = of 4]{};
\end{tikzpicture}  } } \\[.25in]
\text{Group 2: } & 
\parbox[c]{\hsize}{{\color{black} \scalebox{0.5}{
 \begin{tikzpicture}[node distance=0 cm,outer sep = 0pt]
        \node[bhor] (1) at ( 0,   0) {};
        \node[bhor] (2) [below = of 1]{} ;
        \node[bver] (3) at (1.5, -0.5){};         
        \node[bver] (4) [right = of 3] {};        
        \node[bver] (5) [right = of 4]{} ; 
\end{tikzpicture} }   \qquad
\scalebox{0.5}{
\begin{tikzpicture}[node distance=0 cm,outer sep = 0pt]
        \node[bver] (1) at ( 0,   0) {};
        \node[bhor] (2) at (1.5, 0.5){} ;
        \node[bhor] (3) [below = of 2]{} ;         
        \node[bver] (4) at (3, 0) {};        
        \node[bver] (5) [right = of 4] {}; 
\end{tikzpicture}  } \qquad
\scalebox{0.5}{
\begin{tikzpicture}[node distance=0 cm,outer sep = 0pt]
        \node[bver] (1) at ( 0,   0) {};
        \node[bver] (2) [right = of 1] {};
        \node[bhor] (3) at (2.5, 0.5) {};         
        \node[bhor] (4) [below = of 3]{} ;        
        \node[bver] (5) at (4, 0) {}; 
\end{tikzpicture}  } \qquad
\scalebox{0.5}{
\begin{tikzpicture}[node distance=0 cm,outer sep = 0pt]
        \node[bver] (1) at ( 0,   0) {};
        \node[bver] (2) [right = of 1] {};
        \node[bver] (3) [right = of 2] {};         
        \node[bhor] (4) at (3.5, 0.5) {};        
        \node[bhor] (5) [below = of 4]{}; 
\end{tikzpicture}  } } }\\[.25in]
\text{Group 3: } & 
\parbox[c]{\hsize}{{\color{black}\scalebox{0.5}{
\begin{tikzpicture}[node distance=0 cm,outer sep = 0pt]
        \node[bhor] (1) at ( 0,   0) {};
        \node[bhor] (2) [below = of 1] {};
        \node[bhor] (3) [right = of 1] {};         
        \node[bhor] (4) [below = of 3] {};        
        \node[bver] (5) at (3.5, -0.5) {}; 
\end{tikzpicture}  }  \qquad
\scalebox{0.5}{
\begin{tikzpicture}[node distance=0 cm,outer sep = 0pt]
        \node[bver] (1) at ( 0,   0) {};
        \node[bhor] (2) at (1.5, 0.5) {};
        \node[bhor] (3) [right = of 2] {};         
        \node[bhor] (4) [below = of 2] {};        
        \node[bhor] (5) [below = of 3] {}; 
\end{tikzpicture}  }}} \\[.25in]
\text{Group 4: } & 
\parbox[c]{\hsize}{\scalebox{0.5}{
\begin{tikzpicture}[node distance=0 cm,outer sep = 0pt]
        \node[bhor] (1) at ( 0,   0) {};
        \node[bhor] (2) [below = of 1] {};
        \node[bver] (3) at (1.5, -0.5) {};         
        \node[bhor] (4) at (3, 0) {};        
        \node[bhor] (5) [below = of 4] {}; 
\end{tikzpicture}  }}
\end{tabularx}
\smallskip

\noindent The Catalan number $C_n$ counts increasing YT 
of an $n$-stack when the set of labels is fixed.
Group 1 is the unique tiling of $(5^2)$ with no horizontal stacks, so there is $C_0=1$ domino tableau associated with this tiling. Group 2 has the four tilings of $(5^2)$ that have a single $1$-stack, each of which has $C_1=1$ possible labeling. Group 3 consists of the two tilings having a 2-stack. Each of these tilings has $C_2=2$ labelings. Finally the tiling in Group 4 has two $1$-stacks, and it is notably different from the others because it has two disjoint stacks. This tiling has $C_1^2=1$ corresponding labeling. We enumerate the domino tableaux of shape $(5^2)$ using these sets, multiplying each Catalan expression by the number of unique tilings associated with it and summing these products: $C_0+4C_1+2C_2+C_1^2= 1+4\cdot1+2\cdot2+1=10$. 
\end{example}

The following result generalizes this idea for counting domino tableaux of shape $(n^2)$.

\begin{theorem}\label{CatForm}
For $n\geq 1$,    
$$\#\DT(\nn)= \displaystyle \sum_{\alpha} \binom{n-2j+1}{\ell} C_\alpha,$$
where the sum runs over all compositions of $j$ with $\ell$ parts such that
$2j+\ell-1\leq n$, and where $C_\alpha=C_{\alpha_1} \cdots C_{\alpha_\ell}$. 
\end{theorem}

\begin{proof}
    By Lemma~\ref{Cat}, given a $k$-stack, there are $C_k$-many ways to label it. Consider an arbitrary domino tiling. Multiply the corresponding Catalan numbers for each stacking in the tiling to obtain $C_\alpha=C_{\alpha_1}\cdots C_{\alpha_\ell}$, where $\alpha_i$ is the number of stacks in the $i$th stacking. 
    
    The condition $2j+\ell-1\leq n$ is because each stack counts two boxes horizontally (so $2j$ in total) and between distinct stackings there needs to be at least one vertical domino.
    
    Now we want to determine the multiplicity for each $C_\alpha$, that is, how many tilings have $C_\alpha$-many labelings.  Put another way, we want to determine how many ways we can tile the shape $2\times n$ with dominoes such that we get, from left to right, an $\alpha_1$-stack, then an $\alpha_2$-stack, etc. If $\alpha$ is a composition of $j$, then the stackings account for $2j$ horizontal dominoes. Hence, there will be $n-2j$ vertical dominoes in the tableau. There must be at least one vertical domino between two stackings. Therefore, if $\ell$ is the length of the composition $\alpha$, then we choose any $\ell$ locations from before, after, and among the $n-2j$ vertical dominoes, into which we place the stackings. There are $\binom{n-2j+1}{\ell}$ ways to do this.  
    
    The total number of domino tableaux follows from summing over all compositions of $j$.
\end{proof}

Combining Theorems~\ref{BIJ} and~\ref{CatForm} yields the following Catalan identity.

\begin{corollary}
For $n\geq 1$,
$$\displaystyle \binom{n}{\lfloor n/2 \rfloor} =  \sum_{\alpha} \binom{n-2j+1}{\ell} C_\alpha, $$
  where the sum runs over all compositions of $j$ with $\ell$ parts such that $2j+\ell-1\leq n$.
\end{corollary}

We demonstrate this result with an example.

\begin{example} 
For $n=12$, after combining terms with the same product of Catalan numbers, we have 
    \begin{align*}
    \binom{12}{6} &= \binom{13}{0}C_0+ \binom{11}{1}C_1+\binom{9}{1}C_2+ \binom{7}{1}C_3+\binom{5}{1}C_4+\binom{3}{1}C_5+\binom{1}{1}C_6 \\
    &+\binom{9}{2}C_1C_1 +2\binom{7}{2}C_1C_2+2\binom{5}{2}C_1C_3+2\binom{3}{2}C_1C_4+ \binom{5}{2}C_2C_2+2\binom{3}{2}C_2C_3\\
    &+\binom{7}{3}C_1C_1C_1+3\binom{5}{3}C_1C_1C_2+3\binom{3}{3}C_1C_1C_3+3\binom{3}{3}C_1C_2C_2 \\
    &+\binom{5}{4}C_1C_1C_1C_1 =924.
    \end{align*}
\end{example}

Notice that in the proof of Theorem~\ref{CatForm}, the $C_\alpha$ term accounts for the number of labelings for a given tiling, and the remainder of the right-hand side counts the possible tilings with $C_\alpha$ labelings. Therefore, if we consider the right-hand side without the Catalan product, it counts the domino tilings of shape $2\times n$. Thus, we get the following Fibonacci identity.

\begin{corollary}
    Let $F_n$ be the $n$th Fibonacci number where $F_0=F_1=1$. Then 
    \begin{equation*}
        F_n= \sum_{\alpha} \binom{n-2j+1}{\ell},
    \end{equation*}
    where the sum runs over all compositions of $j$ with $\ell$ parts such that $2j+\ell-1\leq n$.
\end{corollary}

\section{Cyclic sieving on $\DT(\nn)$}\label{all CSP}
\subsection{The cyclic sieving function} \label{CSP function}

The next step towards proving a CSP on the set $\DT(\nn)$ is to determine the polynomial $f(q)$ of the triple. We look at the $q$-analogue of the formula enumerating $\DT(\nn)$, and define 
\begin{equation*}
   f(q): =\displaystyle \qbin{n}{\lfloor n/2 \rfloor}_q.
\end{equation*}

Sagan defined the $q$-binomial coefficient as the generating function of the inversion set of $W_{n,k}$, the set of binary words of length $n$ with $k$ zeros~\cite{Sag-q}. Since inversions and the major index are both Mahonian statistics, we get that
\begin{equation*}
    \qbin{n}{k}_q = \displaystyle \sum_{w\in W_{n,k}} q^{\mathrm{inv}(w)} =\displaystyle \sum_{w\in W_{n,k}} q^{\maj(w)}.
\end{equation*}
We now work to show that the outer equality holds for the case where $k=\lfloor n/2 \rfloor$ and the sum runs over $\DT(\nn)$. 

\begin{theorem}\label{word}
    There is a bijection between the sets $\DT(\nn)$ and $W_{n,\lfloor n/2\rfloor}$.
\end{theorem}

\begin{proof}
    Recall the bijection $\Phi$ from the proof of Theorem~\ref{BIJ}. Fix $D\in \DT(\nn)$. Construct an element of $W_{n,\lfloor n/2\rfloor}$ as follows: the $i$th letter of the word is $0$ if and only if $i \in \Phi(D)$. 
    
    Given $a_1 \cdots a_n\in W_{n,\lfloor n/2\rfloor}$, define $S = \{i \mid a_i = 0\}$, an $\lfloor n/2 \rfloor$-element subset of $[n]$. By Theorem~\ref{BIJ}, this uniquely determines a domino tableau $\Phi^{-1}(S)$.
\end{proof}

The bijection of Theorem~\ref{word} also preserves the descent statistic, which subsequently preserves the major index. To see this, we present an alternative way of determining the descents of a domino tableau using its \emph{reduced subset}. 

\begin{definition}
    Let $D$ be a domino tableau, and set $S:= \Phi(D)$. The \emph{reduced subset} of $S$ is the subset that contains only the values of bottom horizontal dominoes. 
\end{definition}

\begin{example}
    For $D= \begin{matrix}
        \scalebox{0.5}{
\begin{tikzpicture}[node distance=0 cm,outer sep = 0pt]
        \node[bver] (1) at ( 0,   0) {\bf \Large 1};
        \node[bhor] (2) at (1.5, 0.5) {\bf \Large 2};
        \node[bhor] (3) [below = of 2] {\bf \Large 3};         
        \node[bhor] (4) [right = of 2] {\bf \Large 4};        
        \node[bhor] (5) [below = of 4] {\bf \Large 5}; 
        \node[bver] (6) at (5,   0) {\bf \Large 6};
\end{tikzpicture}  }
    \end{matrix}$, we have $S = \{ 1, 3, 5\}$, with reduced subset $\{ 3, 5\}$.
\end{example}

Recalling Definition~\ref{defn:descent}, we see that in a domino tableau of shape $(\nn)$, the only candidates for descents are the labels of top horizontal dominoes. Specifically, a candidate must be one less than the label of a bottom horizontal domino, which is an element in the reduced subset. Therefore, we only need to look at the reduced subset for possible descents. 

\begin{lemma}
    Given $D\in \DT(\nn)$, let $R$ be the reduced subset of $\Phi(D)$. For any $i\in R$, we have that $i-1 \in \Des(D)$ if and only if $i-1 \notin R$. 
\end{lemma}

\begin{proof}
    By definition, if $i-1\in \Des(D)$, then $i-1$ is the label of a top horizontal domino, and therefore $i-1\notin R$.
    
    Conversely, if $i-1\notin R$, then it cannot be the value of a vertical domino. Indeed, there cannot be a bottom horizontal domino to its left labeled $i$, and there can be none to its right because there is no integer $i-1 < x < i$ to label the horizontal domino above it. 
    Therefore, if $i-1\notin R$, then $i-1$ is the value of a top horizontal domino and hence $i-1\in \Des(D)$. 
\end{proof}

We now compare this process for finding descents of $D\in\DT(\nn)$ with the descents of words in $W_{n,\lfloor n/2\rfloor}$.

\begin{definition}
    Given a word $w=a_1 \cdots a_n\in W_{n,\lfloor n/2\rfloor}$, the \emph{descent set} of $w$ is 
    $$\Des(w)= \{ i \mid a_i=1>0=a_{i+1} \}.$$
\end{definition}

One way to think about this is that if there is a $1$ to the left of a $0$, then that corresponding index is in the descent set. 
Given $D\in \DT(n^2)$ and $S=\Phi(D)$, let $w\in W_{n,\lfloor n/2\rfloor}$ be the word given by the bijection outlined in  the proof of Theorem~\ref{word}. Then the zero entries in $w$ are in the set $S$, which means that the bijection from Theorem~\ref{word} preserves the descents. Thus, we have the following.

\begin{corollary}\label{majq}
For $n\geq 1$, we have $\qbin{n}{\lfloor n/2 \rfloor}_q= \displaystyle \sum_{D\in \DT(\nn)} q^{\maj(D)}$.
\end{corollary}

This generating function will serve as the polynomial in our cyclic sieving result.

\subsection{The cyclic sieving action}\label{CSP Action}

We now establish the cyclic action for our CSP on domino tableaux of shape $(\nn)$. Recall the cyclic shift that serves as the cyclic sieving action for $k$-element subsets of $[n]$ (Theorem~\ref{subset CSP}): 
$$\begin{array}{ccl}
 C \times X & \longrightarrow & X \\
 g^k \hspace{0.05cm} \raisebox{0.2ex}{$\scriptscriptstyle\bullet$}\  \ms & \longmapsto & \ms':= \{ (s+k) \bmod n \mid s\in \ms\}.
\end{array} 
$$
This subset action induces a cyclic action on domino tableaux using the bijection $\Phi$. 

\begin{example} \label{action ex} 
When $n=4$, the induced cyclic shift action yields one 4-cycle and one 2-cycle as seen below. Note that the red labels in the domino tableaux correspond to the elements in the image of the tableau under $\Phi$. 
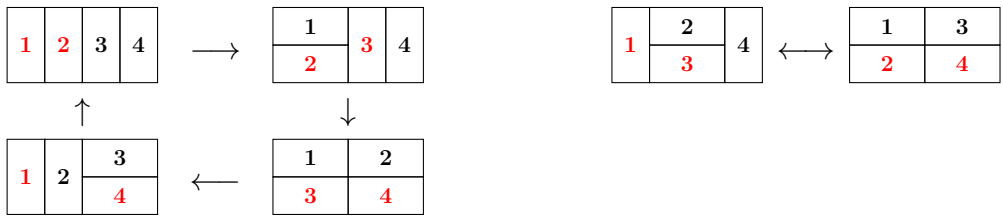
$$
\begin{array}{ccccccc}
\begin{matrix} \scalebox{0.5}{
\begin{tikzpicture}[node distance=0 cm,outer sep = 0pt]
        \node[bver] (1) at ( 0,   0) {\bf \Large {\color{red}1}};
        \node[bver] (2) [right = of 1] {\bf \Large {\color{red}2}};
        \node[bver] (3) [right = of 2] {\bf \Large 3};      
        \node[bver] (4) [right = of 3] {\bf \Large 4};     
\end{tikzpicture}  } \end{matrix}
 & \longrightarrow & 
\begin{matrix} \scalebox{0.5}{
\begin{tikzpicture}[node distance=0 cm,outer sep = 0pt]
        \node[bhor] (1) at ( 0,   0) {\bf \Large 1};
        \node[bhor] (2) [below = of 1] {\bf \Large \color{red}2};
        \node[bver] (3) at (1.5, -0.5) {\bf \Large \color{red}3};     
        \node[bver] (4) [right = of 3] {\bf \Large 4};     
\end{tikzpicture}  } \end{matrix} 
& \qquad \qquad  &
\begin{matrix} \scalebox{0.5}{
\begin{tikzpicture}[node distance=0 cm,outer sep = 0pt]
        \node[bver] (1) at ( 0,   0) {\bf \Large \color{red}1};
        \node[bhor] (2) at (1.5, 0.5) {\bf \Large 2};
        \node[bhor] (3) [below = of 2] {\bf \Large \color{red}3};     
        \node[bver] (4) at (3,0) {\bf \Large 4};     
\end{tikzpicture}  } \end{matrix}  \longleftrightarrow  \begin{matrix} \scalebox{0.5}{
\begin{tikzpicture}[node distance=0 cm,outer sep = 0pt]
       \node[bhor] (1) at ( 0,   0) {\bf \Large 1};
        \node[bhor] (2) [below = of 1] {\bf \Large \color{red} 2};
        \node[bhor] (3) [right = of 1] {\bf \Large 3};      
        \node[bhor] (4) [below = of 3] {\bf \Large \color{red}4};     
\end{tikzpicture}  } \end{matrix} \\
\uparrow & & \downarrow & & & &  \\[0.1cm]
\begin{matrix} \scalebox{0.5}{
\begin{tikzpicture}[node distance=0 cm,outer sep = 0pt]
        \node[bver] (1) at ( 0,   0) {\bf \Large \color{red}1};
        \node[bver] (2) [right = of 1] {\bf \Large 2};
        \node[bhor] (3) at (2.5,0.5) {\bf \Large 3};     
        \node[bhor] (4) [below = of 3] {\bf \Large \color{red}4};     
\end{tikzpicture}  } \end{matrix}  
& \longleftarrow  &
\begin{matrix} \scalebox{0.5}{
\begin{tikzpicture}[node distance=0 cm,outer sep = 0pt]
       \node[bhor] (1) at ( 0,   0) {\bf \Large 1};
        \node[bhor] (2) [right = of 1] {\bf \Large 2};
        \node[bhor] (3) [below = of 1] {\bf \Large \color{red}3};     
        \node[bhor] (4) [below = of 2] {\bf \Large \color{red}4}; 
\end{tikzpicture}  } \end{matrix} 
& & & &  
\end{array}
$$
\end{example}

As suggested in the previous example, the orientation of the domino labeled $x$ in a tableau $D$ is related to the orientation of the domino labeled $x+1$ in the next tableau after $D$ in the cycle. We examine this more precisely below.

Fix $D \in \DT(\nn)$ and set $S := \Phi(D)$. As in the proof of Theorem~\ref{BIJ}, the set $\ms$ can be partitioned as 
$$\ms = H \sqcup V,$$
where $H$ is the set of labels of bottom horizontal dominoes in $\ms$, and $V$ is the set of labels of vertical dominoes in $\ms$. Equivalently, the set $H$ is the reduced subset of $\ms$, and $V = \ms \setminus H$. Similarly, we write
$$\ms' = H' \sqcup V'$$ where $\ms'$ is the subset of the new domino tableau obtained from the cyclic action.
Finally,  for any $r \in \mathbb{R}$ and a set $U$ of real numbers, define 
$$U_s = \{r \in U \mid r < s\}.$$

Note that $n$ can never be in $V$ because if $|H| = k$, then there are $n-2k$ vertical dominoes, and the first $\lfloor n/2 \rfloor - k$ of them would not include the last domino. Similarly, $1$ can never be in $H$ because there is no smaller-labeled domino that could sit above it. Thus if $n \in \ms$ then $n \in H$, and if $1 \in \ms$ then $1 \in V$.

We want to prove that if $s < n$ and $s \in H$, then $s+1 \in H'$. We also want to describe the circumstances under which $s \in V$ leads to $s+1 \in H'$. We begin by characterizing membership in $H$ and $V$, which follows by rephrasing Theorem~\ref{BIJ}.

\begin{lemma}\label{lem:inequality decides H/V}
    For any $s \in \ms$,
    $$2|H_s| + |V_s| \le s-1.$$
    Moreover,
    $$H = \{s \in \ms \mid 2|H_s| + |V_s| < s-1\} \text{ and } V = \{s \in \ms \mid 2|H_s| + |V_s| = s-1\}.$$
\end{lemma}

Now, we characterize what happens to an element of $\ms$ when it is acted on by the generator of the cyclic group.
\begin{theorem}\label{thm: H stays in H, V rarely moves to H}
    Fix $s < n$.
    \begin{enumerate}[label = (\alph*)]
        \item If $s \in H$, then $s+1 \in H'$.
        \item If $s \in V$, then
        $$\begin{cases}
            s+1 \in H' & \text{if and only if } V_s = \emptyset \text{ and } n \not\in \ms, \text{ and}\\
            s+1 \in V' & \text{otherwise.}
        \end{cases}$$
    \end{enumerate}
\end{theorem}

\begin{proof}
    We prove the result by induction on $s$. For the base case, let $s<n$ be the minimal element of $\ms$; in other words, $V_s = H_s = \emptyset$.
    \begin{enumerate}[label = (\alph*)] 
        \item If $s \in H$, then necessarily $s \ge 2$. Furthermore, $H_{s+1}'=\emptyset$ if $n \not\in \ms$, and $H_{s+1}'=\{1\}$ if $n \in \ms$. Therefore
        $$2|H_{s+1}'| + |V_{s+1}'| \le 2\cdot 0 + 1 = 1 < 2 \le s.$$
        Thus, by Lemma~\ref{lem:inequality decides H/V}, we have $s+1 \in H'$.
        \item If $s \in V$, then Lemma~\ref{lem:inequality decides H/V} requires $s = 1$. Therefore $2 \in H'$ if and only if $1 \in \ms'$ (equivalently, $n \in \ms$). Otherwise, $2 \in V'$.
    \end{enumerate}

    Now consider a non-minimal element $s \in \ms$ with $s < n$, and assume that the result holds for all elements $r\in \ms$ with $r <s$.

    First suppose that $s$ is minimal in $V$ (i.e., $s \in V$ and $V_s = \emptyset$). By the inductive hypothesis, $\{r+1 \mid r \in H_s\} \subset H'$. Then $s+1\in V'$ if and only if $\{r+1 \mid r \in H_s\}=H_{s+1}'$. By Lemma~\ref{lem:inequality decides H/V}, this means that $2|H'_{s+1}| + |V'_{s+1}| = s$ if and only if $|V'_{s+1}| = 1$, which happens if and only if $1 \in V'$. Thus, when $V_s = \emptyset$, we have $s+1 \in V'$ if and only if $1 \in V'$ (i.e., if and only if $n \in \ms$).

    Now suppose that $s \in \ms$ is not the minimal element of $V$ (i.e., it is non-minimal in $V$ or it is in $H$). By the inductive hypothesis, we again have $\{r+1 \mid r \in H_s\} \subset H'$. Note that at most one $r \in V_s$ corresponds to $r+1 \in H'_{s+1}$. Moreover, there is such an $r$ (namely, the minimal element of $V$) if and only if $n \not\in \ms$. Therefore
    $$2|H_{s+1}'| + |V_{s+1}'| = \begin{cases}
        2|H_s| + (|V_s| + 1)& \text{if } n \in \ms, \text{ and}\\
        2(|H_s| + 1) + (|V_s| - 1) & \text{if } n \not\in \ms.
    \end{cases}$$
    Either way, $2|H_{s+1}'| + |V_{s+1}'| = 2|H_s| + |V_s| + 1$. Thus, using Lemma~\ref{lem:inequality decides H/V}, we see that if $s \in V$ is non-minimal then $s+1 \in V'$, and if $s \in H$ then $s+1 \in H'$.
\end{proof}

\begin{corollary}\label{cor:at most one V becomes H}
    There is at most one $s \in V$ with $s+1 \in H'$, and it exists if and only if $n \not\in \ms$. Moreover, if such an $s$ exists, then this $s$ is odd and it is the minimal value of $V$.
\end{corollary}

Theorem~\ref{word} showed that domino tableaux in $\DT(\nn)$ are in bijection with elements of $W_{n,\lfloor n/2 \rfloor}$, which are binary words of length $n$ having $\lfloor n/2 \rfloor$ zeros. We can think of our cyclic action in terms of how it affects the corresponding binary word associated with a domino tableau. 

\begin{example}\label{ex: words}
    The set $W_{4,2}$ decomposes into one 4-cycle and one 2-cycle. 
$$
\begin{array}{ccccccc}
{\color{red}{0011}} & \longrightarrow & 1001 
& \qquad \qquad &
{\color{red}01}01 & \longleftrightarrow & 1010 \\
\uparrow & & \downarrow & & & & \\
0110 & \longleftarrow & 1100 & & & 
\end{array}
$$
This example aligns with the orbits of domino tableaux from Example~\ref{action ex}. 
\end{example}

The equivalence classes of binary words of length $n$ with $k$ zeros under cyclic shift are known as \emph{necklaces}, and we can represent each class by its lexicographically smallest word. The unique aperiodic subword that generates this representative is known as a \emph{binary Lyndon word}. We denote by $L_{n,k}$ the binary Lyndon words of length $n$ with $k$ zeros, and set $\ell_{n,k} := |L_{n,k}|$. For instance, the words highlighted in red in Example~\ref{ex: words} are binary Lyndon words: $L_{2,1} = \{01\}$ and $L_{4,2} = \{0011\}$.  
We have the following result connecting domino tableaux in $\DT(n^2)$ with necklaces.

\begin{lemma}\label{lyndon}
    For even values of $n$, the orbits of $C$ acting on $\DT(\nn)$ are in bijection with the necklaces of binary words under the cyclic shift action on $W_{n,n/2}$.
\end{lemma}

\subsection{The CSP on $\DT(\nn)$}\label{CSP Proof}
We can now state and prove our new instance of the cyclic sieving phenomenon. 

\begin{theorem}
    Let $C=\mathbb{Z}_n$ act on $\DT(\nn)$ via the conjugation of the cyclic shift action induced by the bijection $\Phi$. 
    Then the triple $\big( \DT(\nn), C, f(q) \big)$ exhibits the cyclic sieving phenomenon where $f(q)= \displaystyle \sum_{D\in \DT(\nn)} q^{\maj(D)}$.
\end{theorem}

\begin{proof}
Let $g\in C$ denote a generator of the cyclic group. 

We will show that 
\begin{equation}\label{CSPP}
       \# \{ D\in \DT(\nn) \mid g^k(D)=D\}= f(e^{2\pi ik/n}),
\end{equation}
for all possible powers of $k$ by direct computation. Without loss of generality, we may assume that $k$ divides $n$, i.e., $n=dk$. First, we compute the left-hand side via two cases, depending on the parity of $n$.

 \begin{itemize}
     \item Let $n$ be odd.  Then $n=dk$ with $d$ and $k$ both odd. 
     Fix $D \in \DT(\nn)$ and set $\ms := \Phi(D)$, with orbit size $k$. 
Since $n$ is odd, we have $|\ms|= \lfloor n/2 \rfloor = (dk-1)/2$. We can also partition $\ms$ into $d$-element subsets of the form $\{i + rk : r \in [0,d-1]\}$ for some $i \in [1,k]$.  Therefore, we can also write the size of $\ms$ as $|\ms |= ad$ for some $a$. Then $ad= (dk-1)/2$, which we can rewrite as $1=d(k-2a)$. This can only be true when $d=1$ and $k=2a+1$. Hence, $k=n$.
Thus, when $n$ is odd,
\begin{equation*}
    \# \{ D\in \DT(\nn) \mid g^k(D)=D\}=
\begin{cases}
       \hspace{6.7mm} 0 & \text{if } k<n, \text{ and}\\
       \displaystyle \binom{n}{\lfloor n/2 \rfloor} & \text{if } k=n.
\end{cases}
\end{equation*}

\item Let $n = 2kd$ be even. First, we show that all orbit sizes must be even. Suppose $\ms = \Phi(D)$ for some $D\in \DT(\nn)$ with orbit size $k$. Since $n$ is even, we have that $|\ms|= \lfloor n/2 \rfloor = n/2 =k d$. We can 
again partition $\ms$ into $2d$-element subsets of the form $\{i + rk : r \in [0,2d-1]\}$ for some $i \in [1,k]$. Therefore, we have $|\ms|=2d a$ for some $a$. Then $2d  a= k d$ and $2a=k$, so the orbit size is even.

By Lemma~\ref{lyndon}, we know that each orbit of size $i$ has a Lyndon word representative in $L_{i,i/2}$. 
Thus, the number of domino tableaux in an orbit of size $i$ is $\ell_{i,i/2}\cdot i$. To enumerate all domino tableaux such that $g^k(D)=D$, we therefore take the sum $\sum \ell_{i,i/2}\cdot i$ over all $i$ that are even divisors of $k$ (including $k$ itself). This argument also counts the binary words of length $k$ with $k/2$ zeros. 
Hence 
\begin{equation*}
    \displaystyle \sum_{i} \ell_{i,i/2}\cdot i=\binom{k}{k/2}.
\end{equation*}

This proves that, when $n$ is even, 
\begin{equation*}
    \# \{ D\in \DT(\nn) \mid g^k(D)=D\}=
\begin{cases}
       \hspace{5.5 mm} 0 & \text{if } k \text{ is odd, and} \\
       \displaystyle \binom{k}{k/2 } & \text{if } k \text{ is even.}
\end{cases}
\end{equation*}
\end{itemize}

Now we compute the right-hand side of Equation~\eqref{CSPP}. By Corollary~\ref{majq}, we know that $\sum_{D\in \DT(\nn)} q^{\maj(D)}= \qbin{n}{\lfloor n/2 \rfloor}_q$. In~\cite[Proposition 4.2(ii)]{OGCSP}, the authors give a way to evaluate the $q$-binomial coefficient at roots of unity: given $n,d\in \mathbb{Z}^+$ such that $n=dk$, 
\begin{equation*}
    \qbin{n}{\lfloor n/2 \rfloor}_{q=e^{2\pi ik/n}}= \binom{n/d}{\lfloor n/2\rfloor /d}.
\end{equation*}
Note that the binomial coefficient evaluates to zero if it has a non-integer argument. We again proceed by cases.

\begin{itemize}

\item Suppose $d=1$. Then $k=n$ and $\displaystyle \qbin{n}{\lfloor n/2 \rfloor}_{q=e^{2\pi i}} = \binom{n}{\lfloor n/2 \rfloor}$. Therefore, when $k=n$, we get back the original binomial coefficient, which aligns with the computation on the left-hand side.

\item Suppose $n$ is odd and $d\neq 1$. Then $ \qbin{n}{\lfloor n/2 \rfloor}_{q=e^{2\pi ik/n}}= \displaystyle  \binom{n/d}{\lfloor n/2\rfloor /d}$. Since $n$ is odd, we have $\lfloor n/2 \rfloor= (n-1)/2$. Therefore, $\lfloor n/2 \rfloor/d= (n-1)/(2d)$. Given that $d\neq 1$, it cannot be a divisor of both $n$ and $n-1$. Hence, $\lfloor n/2 \rfloor/d$ is not an integer, and the binomial coefficient evaluates to 0.

\item Suppose $n$ is even and $d\neq 1$. Then we can rewrite $d=n/k$ and $\lfloor n/2 \rfloor=n/2$, so we have that
$\qbin{n}{\lfloor n/2 \rfloor}_{q=e^{2\pi ik/n}}= 
\displaystyle \binom{k}{k/2}$. When $k$ is even, this binomial coefficient is equal to the left-hand side of Equation~\eqref{CSPP}. When $k$ is odd, $k/2$ is not an integer, and thus the binomial coefficient evaluates to 0. 
\end{itemize}

Hence Equation~\eqref{CSPP} holds, and the triple exhibits the cyclic sieving phenomenon. 
\end{proof}

\section{Enumerating domino tableaux of shape $(n^m)$} 
\label{count rectangles}

Next, we extend the enumeration result to domino tableaux of shape $m\times n$. For that, we start by noticing that at least  one of $m$ or $n$ must be even. Moreover, since reflecting $D \in \DT(n^m)$ across the diagonal through its upper left corner will produce $D' \in \DT(m^n)$, we have the following symmetry.

\begin{lemma}\label{lem:rectangle symmetry}
    For all positive integers $m$ and $n$, we have $\# \DT(n^m)=\# \DT(m^n)$.
\end{lemma}

To be tileable by dominoes, a rectangle must have even area. Without loss of generality, thanks to Lemma~\ref{lem:rectangle symmetry}, we can assume that $m=2k$ is even. 

\begin{theorem} For any positive integers $k$ and $n$,
\begin{equation}\label{2kn}
    \#\DT(n^{2k})= 
    \dfrac{(k \lceil n/2\rceil)! \hspace{1 mm} \cdot \hspace{1 mm} (k \lfloor n/2\rfloor)!}{ \displaystyle \prod_{i=1}^k i(i+1)\cdots \left(i+\left\lceil \frac{n}{2}\right\rceil -1\right) \cdot i(i+1)\cdots \left(i +\left\lfloor \frac{n}{2}\right\rfloor-1\right)} \cdot \binom{kn}{k \lfloor n/2 \rfloor}.
\end{equation}
\end{theorem}

\begin{proof}
Using the map $\Gamma$ from Theorem~\ref{DTSYT}, every $D \in \DT(n^{2k})$ has a unique corresponding pair of Young tableaux of shapes $(\mu, \nu)$ with the same labeling set as $D$. We use $\Gamma$ to determine the shapes of $\mu$ and $\nu$. For any $k$ and $n$, we can consider the domino tableau of all vertical dominoes with labels $1, \dots, n$ covering the first two rows,
$n+1, \dots, 2n$ covering the next two rows, etc. Looking at the diagonals of this tableau and how they intersect the dominoes, the dominoes in odd columns are Type I, and the dominoes in even columns are Type II. Separating the types into their respective Young tableaux, we get two rectangular shapes, specifically $(\mu, \nu  )=  (k\times \lceil n/2 \rceil, k\times \lfloor n/2\rfloor)$. Therefore, $\#\DT(n^{2k})$ is equal to the number of pairs of increasing YT of shapes $(\mu, \nu)=  (k\times \lceil n/2 \rceil, k\times \lfloor n/2\rfloor)$.

To compute the number of pairs $(\mu, \nu)$, we first determine the number of ways to distribute the labels of $D\in \DT(n^{2k})$ between $\mu$ and $\nu$. In particular, we count the possible labeling sets of $\nu$. Elements of $\DT(n^{2k})$ are labeled with the numbers in $[kn]$. Thus, there are $kn$ total numbers to choose from when labeling $\nu$, and we pick $k\cdot \lfloor n/2 \rfloor$ of them, in $\binom{kn}{k \lfloor n/2 \rfloor}$ ways. The remaining $k\lceil n/2 \rceil$ numbers in $[kn]$ serve as the labels of $\mu$. 

The number of increasing YT of shape $\lambda \vdash n$ and with a fixed labeling set of size $n$ is equal to the number of standard YT of shape $\lambda$. This is because all the entries are different and so we can map them in an order-preserving way to the set $[n]$. (This type of maneuver can be called ``flattening.'') 
Thus, the desired enumerations are given by the hook length formula and we have that
$$ f^\mu= \dfrac{(k \lceil n/2\rceil)!}{\prod h_{\mu}(i,j)} \qquad \text{and} \qquad f^\nu= \dfrac{(k \lfloor n/2\rfloor)!}{\prod h_{\nu}(i,j)}.$$

Since $\mu$ and $\nu$ both have rectangular shape, we can determine the hook lengths of the boxes in these tableaux in terms of $k$ and $n$:
$$ f^\mu= \dfrac{(k\cdot \lceil n/2\rceil)!}{\displaystyle\prod_{i=1}^k i(i+1)(i+2)\cdots (i+\lceil n/2\rceil-1)} \text{ \ and \ } f^\nu= \dfrac{(k\cdot \lfloor n/2\rfloor)!}{\displaystyle\prod_{i=1}^k i(i+1)(i+2)\cdots (i+\lfloor n/2\rfloor-1)}.
$$

Multiplying $f^\mu$ and $f^\nu$ by the binomial coefficient counting the choices of labeling sets, we get the desired equality. 
\end{proof}

We demonstrate this enumeration for the shape $6\times 5$.

\begin{example}
    To enumerate $\DT(5^6)$, we first determine the shapes of $\mu$ and $\nu$ using $\Gamma$.
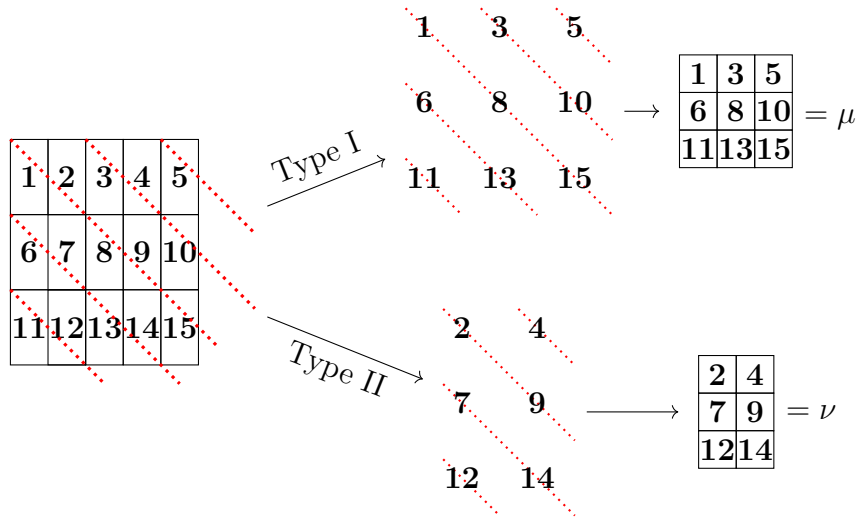
\begin{eqnarray*}
\newdimen\squarewidth  \squarewidth=15pt
\begin{tikzpicture}
\node (dt) at (0,0) {
\begin{tikzpicture}
    \draw (0,0) rectangle (0.5,-1);
    \node at (0.25, -0.5) {$\mathbf{1}$};
    \draw (0.5,0) rectangle (1,-1);
    \node at (0.75, -0.5) {$\mathbf{2}$};
    \draw (1,0) rectangle (1.5,-1);
    \node at (1.25, -0.5) {$\mathbf{3}$};
    \draw (1.5,0) rectangle (2,-1);
    \node at (1.75, -0.5) {$\mathbf{4}$};
    \draw (2,0) rectangle (2.5,-1);
    \node at (2.25, -0.5) {$\mathbf{5}$};
    \draw (0,-1) rectangle (1,-2);
    \node at (0.25, -1.5) {$\mathbf{6}$};
    \draw (1,-1) rectangle (0.5,-2);
    \node at (.75, -1.5) {$\mathbf{7}$};
    \draw (1,-1) rectangle (1.5,-2);
    \node at (1.25, -1.5) {$\mathbf{8}$};
    \draw (1.5,-1) rectangle (2,-2);
    \node at (1.75, -1.5) {$\mathbf{9}$};
    \draw (2,-1) rectangle (2.5,-2);
    \node at (2.25, -1.5) {$\mathbf{10}$};
    \draw (0,-2) rectangle (1,-3);
    \node at (0.25, -2.5) {$\mathbf{11}$};
    \draw (1,-2) rectangle (0.5,-3);
    \node at (.75, -2.5) {$\mathbf{12}$};
    \draw (1,-2) rectangle (1.5,-3);
    \node at (1.25, -2.5) {$\mathbf{13}$};
    \draw (1.5,-2) rectangle (2,-3);
    \node at (1.75, -2.5) {$\mathbf{14}$};
    \draw (2,-2) rectangle (2.5,-3);
    \node at (2.25, -2.5) {$\mathbf{15}$};
    \draw[dotted, red, very thick] (0,0) -- (2.75,-2.75);
    \draw[dotted, red, very thick] (1,0) -- (3.25,-2.25);
    \draw[dotted, red, very thick] (2,0) -- (3.25,-1.25);
    \draw[dotted, red, very thick] (0,-1) -- (2.25,-3.25);
    \draw[dotted, red, very thick] (0,-2) -- (1.25,-3.25);
\end{tikzpicture}}; 
\node (diag-mu) at (5,2) {
\begin{tikzpicture}
    \draw[dotted, red,  thick] (0,0) -- (2.75,-2.75);
    \draw[dotted, red,  thick] (1,0) -- (2.75,-1.75);
    \draw[dotted, red,  thick] (2,0) -- (2.75,-0.75);
    \draw[dotted, red,  thick] (0,-1) -- (1.75,-2.75);
    \draw[dotted, red,  thick] (0,-2) -- (.75,-2.75);
    \node at (0.25, -0.25) {$\mathbf{1}$};
    \node at (1.25, -0.25) {$\mathbf{3}$};
    \node at (2.25, -0.25) {$\mathbf{5}$};
    \node at (0.25, -1.25) {$\mathbf{6}$};
    \node at (1.25, -1.25) {$\mathbf{8}$};
    \node at (2.25, -1.25) {$\mathbf{10}$};
    \node at (0.25, -2.25) {$\mathbf{11}$};
    \node at (1.25, -2.25) {$\mathbf{13}$};
    \node at (2.25, -2.25) {$\mathbf{15}$};
\end{tikzpicture}};
\node (diag-nu) at (5,-2) {
\begin{tikzpicture}
    \draw[dotted, red,  thick] (0,0) -- (1.75,-1.75);
    \draw[dotted, red,  thick] (1,0) -- (1.75,-0.75);
    \draw[dotted, red,  thick] (0,-1) -- (1.75,-2.75);
    \draw[dotted, red,  thick] (0,-2) -- (.75,-2.75);
    \node at (0.25, -0.25) {$\mathbf{2}$};
    \node at (1.25, -0.25) {$\mathbf{4}$};
    \node at (0.25, -1.25) {$\mathbf{7}$};
    \node at (1.25, -1.25) {$\mathbf{9}$};
    \node at (0.25, -2.25) {$\mathbf{12}$};
    \node at (1.25, -2.25) {$\mathbf{14}$};
\end{tikzpicture}};
\node (mu) at (8.5,2) {
\begin{tikzpicture}
    \draw (0,0) rectangle (0.5,-0.5);
    \draw (0.5,0) rectangle (1,-0.5);
    \draw (1,0) rectangle (1.5,-0.5);
    \draw (0,-0.5) rectangle (0.5,-1);
    \draw (0.5,-0.5) rectangle (1,-1);
    \draw (1,-0.5) rectangle (1.5,-1);
    \draw (0,-1) rectangle (0.5,-1.5);
    \draw (0.5,-1) rectangle (1,-1.5);
    \draw (1,-1) rectangle (1.5,-1.5);
    \node at (0.25, -0.25) {$\mathbf{1}$};
    \node at (0.75, -0.25) {$\mathbf{3}$};
    \node at (1.25, -0.25) {$\mathbf{5}$};
    \node at (0.25, -0.75) {$\mathbf{6}$};
    \node at (0.75, -0.75) {$\mathbf{8}$};
    \node at (1.25, -0.75) {$\mathbf{10}$};
    \node at (0.25, -1.25) {$\mathbf{11}$};
    \node at (0.75, -1.25) {$\mathbf{13}$};
    \node at (1.25, -1.25) {$\mathbf{15}$};
    \node at (2,-.85) {$=\mu$};
\end{tikzpicture}
};  
\node (nu) at (8.5,-2) {\begin{tikzpicture}
    \draw (0,0) rectangle (0.5,-0.5);
    \draw (0.5,0) rectangle (1,-0.5);
    \draw (0,-0.5) rectangle (0.5,-1);
    \draw (0.5,-0.5) rectangle (1,-1);
    \draw (0,-1) rectangle (0.5,-1.5);
    \draw (0.5,-1) rectangle (1,-1.5);
    \node at (0.25, -0.25) {$\mathbf{2}$};
    \node at (0.75, -0.25) {$\mathbf{4}$};
    \node at (0.25, -0.75) {$\mathbf{7}$};
    \node at (0.75, -0.75) {$\mathbf{9}$};
    \node at (0.25, -1.25) {$\mathbf{12}$};
    \node at (0.75, -1.25) {$\mathbf{14}$};
    \node at (1.5,-0.75) {$=\nu$};
\end{tikzpicture}};
\draw (dt) edge[->] node[sloped,above] {Type I} (diag-mu);
\draw (dt) edge[->] node[sloped,below] {Type II} (diag-nu);
\draw (diag-mu) edge[->]  (mu);
\draw (diag-nu) edge[->] (nu);
\end{tikzpicture}
\end{eqnarray*}
Thus, we want to count pairs of increasing Young tableaux of shapes $(3,3)$ and $(3,2)$, respectively. There are $\binom{15}{6}$ ways of distributing the labels between $\mu$ and $\nu$. Using the hook length formula, we determine the number of increasing YT  of each shape, using those labels, to be:
$$ f^\mu= \dfrac{(3\cdot 3)!}{\displaystyle\prod_{i=1}^3 i(i+1)(i+2)}= 42, \qquad \qquad  f^\nu= \dfrac{(3\cdot 2)!}{\displaystyle\prod_{i=1}^3 i(i+1)}= 5.$$
Hence, $\# \DT(5^6)= \binom{15}{6}\cdot 42\cdot 5= 1,\!051,\!050$.
\end{example}

Equation~\eqref{2kn} simplifies to the formula for $\#\DT(\nn)$ given in Theorem~\ref{CatForm}. There are also nice simplifications for $\# \DT(3^{2k})$ and $\#\DT(n^4)$ involving the Catalan numbers:
\begin{align*}
    \# \DT(3^{2k}) &= C_k\cdot \displaystyle \binom{3k}{k}, \text{ and}\\
    \#\DT(n^4) &= C_{\lceil n/2 \rceil}\cdot C_{\lfloor n/2 \rfloor} \cdot \displaystyle \binom{2n}{2\lfloor n/2 \rfloor}.
\end{align*}

\section{Open problems} \label{conclusion}

As with any work on the cyclic sieving phenomenon, there are many interesting directions for further research. Here we suggest a few that are particularly close to the results of this paper.

The bijection $\Gamma$ described in Definition~\ref{defn:Gamma} enabled us to enumerate the set of rectangular domino tableaux of a given dimension. It would be interesting to see if we could use $\Gamma$ to enumerate other shapes of domino tableaux. For example, it is straightforward to do this for domino tableaux of hook shape.

\begin{theorem} For $m,n\in \mathbb{Z}^+$ such that $m+n$ is even, we have
    \begin{equation*}
        \# \DT(n, 1^m)= \dfrac{(\lceil n/2\rceil \cdot \lfloor m/2 \rfloor)!}{(\lceil n/2\rceil + \lfloor m/2 \rfloor)\cdot (\lceil n/2 \rceil -1)!\cdot (\lfloor m/2 \rfloor)!}.
    \end{equation*}
\end{theorem}

\begin{proof} 
There is only one way to tile a hook with dominoes, depending on parity. In particular, we get the following tilings if both $n$ and $m$ are odd or both are even, respectively.
$$
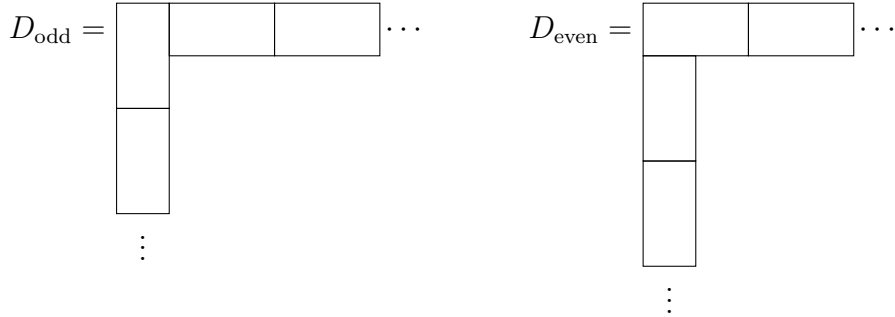
\begin{tikzpicture}[scale=.7]
    
    \draw (0,0) rectangle (1,-2);
    \draw (0,-2) rectangle (1,-4);
    \draw (1,0) rectangle (3,-1);
    \draw (3,0) rectangle (5,-1);
    \draw (5.5,-.5) node {$\cdots$};
    \draw (.5,-4.5) node {$\vdots$};
    \draw (0,-.5) node[left] {$D_{\text{odd}} =$};

    \draw (10,0) rectangle (12,-1);
    \draw (12,0) rectangle (14,-1);
    \draw (10,-1) rectangle (11,-3);
    \draw (10,-3) rectangle (11,-5);
    \draw (14.5,-.5) node {$\cdots$};
    \draw (10.5,-5.5) node {$\vdots$};
    \draw (10,-.5) node[left] {$D_{\text{even}} =$};
\end{tikzpicture}$$
Considering the diagonals of these tilings, all dominoes in $D_{\text{odd}}$ are Type I and all dominoes in $D_{\text{even}}$ are Type II. Therefore, $\# \DT(n, 1^m)$ is equal to the number of increasing YT of shape $\mu= (\lceil n/2 \rceil, 1^{\lfloor m/2 \rfloor})$. The hook length formula gives
$$\# \DT(n, 1^m)= f_{\mu}=\dfrac{(\lceil n/2\rceil \cdot \lfloor m/2 \rfloor)!}{(\lceil n/2\rceil + \lfloor m/2 \rfloor)\cdot (\lceil n/2 \rceil -1)!\cdot (\lfloor m/2 \rfloor)!}.$$
\end{proof}

The enumeration of $\DT(\lambda)$ for arbitrary $\lambda$ is certainly a problem that deserves attention, and the approach of using the hook length formula could be exploited further.

Having established $\# \DT(n^{m})$, we can work to prove a new instance of the cyclic sieving phenomenon on this set. We know there exists a CSP on the set of rectangular standard Young tableaux with promotion and the $q$-analogue of the hook length formula (Theorem~\ref{RTCSP},~\cite{RTCSP}). With this result in mind, we have the following conjecture. 

\begin{conjecture} 
The set of rectangular domino tableaux $\DT(n^{2k})$ exhibits the cyclic sieving phenomenon with function
$$f(q)= 
\dfrac{[k \lceil n/2\rceil]!_q \hspace{1 mm} \cdot \hspace{1 mm} [k \lfloor n/2\rfloor]!_q}{ \displaystyle \prod_{i=1}^k [i]_q[i+1]_q\cdots \left[i+\left\lceil \frac{n}{2}\right\rceil -1\right]_q \cdot [i]_q[i+1]_q\cdots \left[i +\left\lfloor \frac{n}{2}\right\rfloor-1\right]_q} \cdot \displaystyle\qbin{kn}{k \lfloor n/2 \rfloor}_q.
$$.
\end{conjecture}

\bibliographystyle{plain}
\bibliography{References}

\end{document}